\documentclass{article}
\usepackage{amssymb,amsmath}
\usepackage[pdftex]{graphicx}
\usepackage{ dsfont }

\def\qed{\hfill {\hbox{${\vcenter{\vbox{               %HOLLOW SQUARE
   \hrule height 0.4pt\hbox{\vrule width 0.4pt height 6pt
   \kern5pt\vrule width 0.4pt}\hrule height 0.4pt}}}$}}}

\newtheorem{theorem}{Theorem}
\newtheorem{definition}{Definition}
\newtheorem{lemma}[theorem]{Lemma}
\newtheorem{proposition}[theorem]{Proposition}
\newtheorem{corollary}[theorem]{Corollary}
\newtheorem{example}{Example}
\newtheorem{remark}[example]{Remark}

\newenvironment{proof}[1][Proof]{\smallskip\noindent{\bf #1.}\quad}%
{\qed\par\medskip}

\date{}

\title{\Large \textbf{The Theory of Pseudoknots}}

\author{
Allison Henrich
 \footnote{henricha@seattleu.edu, Seattle University, Seattle, WA 98122, United States}\hspace{1cm}
Rebecca Hoberg \footnote{rahoberg@gmail.com, University of Washington, Seattle, WA 98195, United States}\hspace{1cm}
Slavik Jablan\footnote{sjablan@gmail.com, The Mathematical Institute, Belgrade, 11000, Serbia}\\
Lee Johnson\footnote{johns193@seattleu.edu, Seattle University, Seattle, WA 98122, United States}\hspace{1cm}
Elizabeth Minten\footnote{minten@email.sc.edu, University of South Carolina, Columbia, SC 29208, United States}\hspace{1cm}
Ljiljana Radovi\' c\footnote{ljradovic@gmail.com, Faculty of Mechanical Engineering, Ni\v s, 18000, Serbia}}

\begin{document}

\maketitle

\begin{abstract}

Classical knots in $\mathbb{R}^3$ can be represented by diagrams in the plane. These diagrams are formed by curves with a finite number of transverse crossings, where each crossing is decorated to indicate which strand of the knot passes over at that point. A {\em pseudodiagram} is a knot diagram that may be missing crossing information at some of its crossings. At these crossings, it is undetermined which strand passes over. Pseudodiagrams were first introduced by Ryo Hanaki in 2010. Here, we introduce the notion of a {\em pseudoknot}, i.e. an equivalence class of pseudodiagrams under an appropriate choice of Reidemeister moves. In order to begin a classification of pseudoknots, we introduce the concept of a {\em weighted resolution set}, or {\em WeRe-set}, an invariant of pseudoknots. We compute the WeRe-set for several pseudoknot families and discuss extensions of crossing number, homotopy, and chirality for pseudoknots.

\end{abstract}

%\tableofcontents

%%%%%%%%%Introduction%%%%%%%%%%
\section{Introduction}

Recently, Ryo Hanaki introduced the notion of a {\em pseudodiagram} of a knot, link or spatial graph~\cite{hanaki}. A pseudodiagram of a knot is a knot diagram that may be missing some crossing information, as in Figure~\ref{f0}. In other words, at some crossings in a pseudodiagram it is unknown which strand passes over and which passes under. These undetermined crossings are called {\em precrossings}. Special classes of pseudodiagrams are knot diagrams and knot {\em shadows}, i.e. pseudodiagrams containing only precrossings. Pseudodiagrams were originally considered because of their potential to serve as useful models for biological objects related to DNA, but they are interesting objects in their own right.

\begin{figure}[th]
\centerline{\includegraphics[width=3.8in]{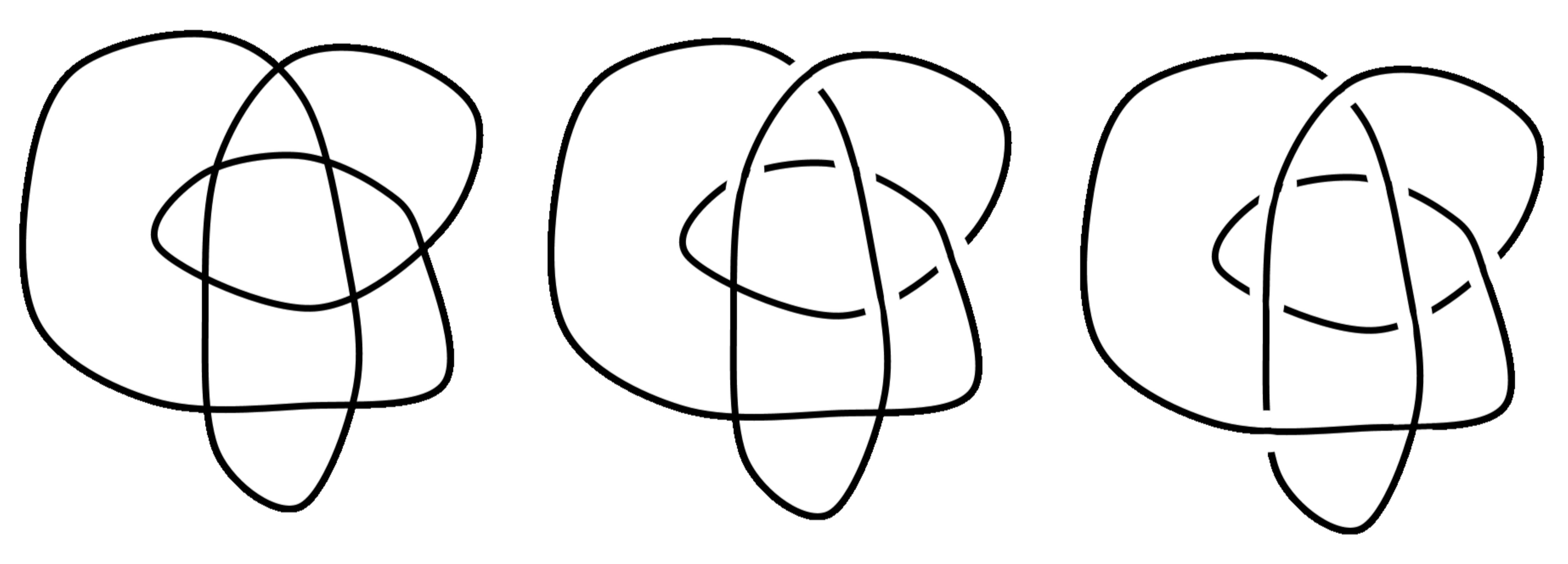}} \vspace*{8pt}
\caption{Examples of pseudodiagrams. \label{f0}}
\end{figure}

Of particular interest in pseudodiagram theory is the computation of the {\em trivializing number} and {\em knotting number} for pseudodiagrams. The trivializing number is the least number of precrossings that must be resolved into crossings in order to produce a pseudodiagram of the unknot. (That is, regardless of how the remaining precrossings are resolved, the unknot is always produced.) Similarly, the knotting number is the least number of precrossings that must be determined to produce a nontrivial knot. Much work has been done to analyze trivializing and knotting numbers~\cite{hanaki},~\cite{SMALL}.

For the purposes of this paper, we are interested in studying the knot theory that arises from considering equivalence classes of pseudodiagrams under equivalence relations generated by a natural set of Reidemeister moves. We refer to these objects as {\em pseudoknots}. Our choice of the set of Reidemeister moves for pseudoknots, pictured in Figure~\ref{rmoves}, was inspired by the theory of singular knots. Singular knots are knots that contain a finite number of self-intersections. If we view precrossings as singular crossings, we recover all of the pseudoknot Reidemeister moves, with the notable exception of the pseudo-Reidemeister I (PR1) move.

\begin{figure}[th]
\centerline{\includegraphics[width=5.8in]{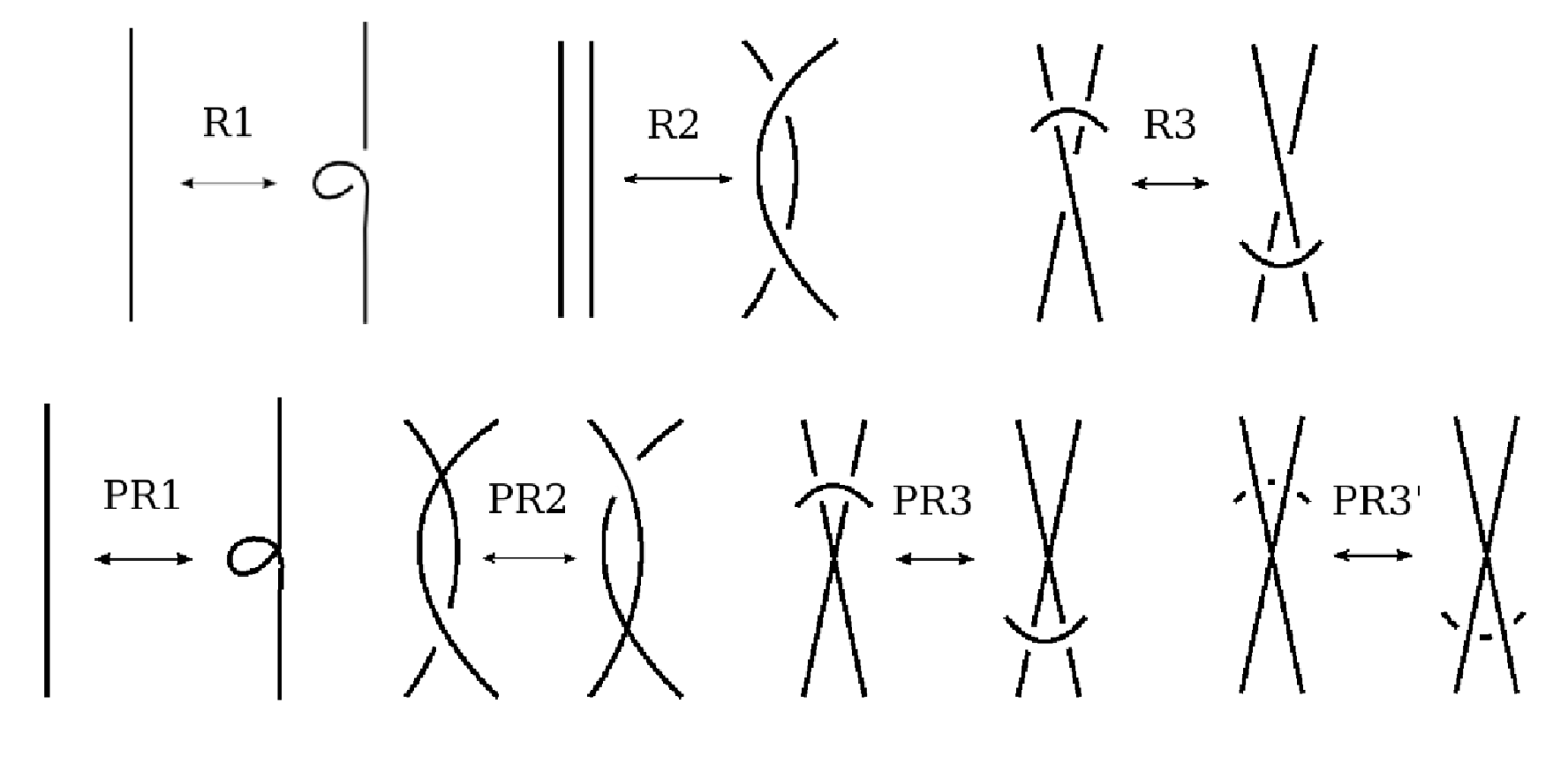}} \vspace*{8pt}
\caption{The pseudo-Reidemeister moves. \label{rmoves}}
\end{figure}

We note that the PR1 move ought to be included for the following reason. Consider a pseudodiagram $P$ and the pseudodiagram $P'$ that is related to $P$ by a PR1 move that introduces a new precrossing, $c$. Given a resolution of the precrossings in $P'$, the canonical corresponding resolution of the precrossings in $P$ produces the same knot type, regardless of how $c$ was resolved in $P'$. So, in an important sense, PR1 preserves the knot type of a pseudodiagram.

Now that we have defined pseudodiagrams, we'd like to learn what we can about their classification. We'd also like to extend several classical knot theoretical notions to our new setting. In Section~\ref{WeRe}, we introduce the primary invariant we use to classify pseudoknots, and in Section~\ref{families}, we discuss the classification of several pseudoknot families. We then turn to the extension of the crossing number in Section~\ref{crossnum}. Section~\ref{hom} is concerned with the homtopy of pseudoknots, and we propose a notion of chirality in Section~\ref{chiral}. We conclude with a few open questions and provide an appendix for tables of pseudoknots with up to five crossings.

%%%%%%%%%%%%WeRe-sets%%%%%%%%%%%%%
\section{Weighted Resolution Sets}\label{WeRe}

As with any knot theory, the primary question in pseudoknot theory asks how we might classify pseudoknots. That is, given any two distinct pseudoknots, how might we prove that they are distinct. One partial answer is to consider the set of all possible knots that can be produced by resolving all precrossings in a diagram of the pseudoknot. A more sophisticated answer is to consider the invariant we call the {\em weighted resolution set} of a pseudoknot.

\begin{definition}  The {\em  weighted resolution set} (or {\em WeRe-set}) of a pseudoknot $P$ is the set of ordered pairs $(K, p_K)$ where $K$ is a resolution of $P$ (i.e. a choice of crossing information for every precrossing) and $p_K$ is the probability that $K$ is obtained from $P$ by a random choice of crossing information, assuming that positive and negative crossings are equally likely. \end{definition}

To illustrate this definition, consider the pseudoknot $T$ in Figure~\ref{part_tref}. There are $2^2=4$ ways to resolve the precrossings in the diagram. Three of the four resolutions result in the unknot, $0_1$, while one resolution results in the trefoil, $3_1$. Thus, the WeRe-set for this example is $\{(0_1, \frac{3}{4}),(3_1,\frac{1}{4})\}$. Note that if we resolve one of the precrossings of $T$ to be a positive crossing, the WeRe-set of the resulting pseudoknot is $\{(0_1, \frac{1}{2}),(3_1,\frac{1}{2})\}$. This shows that the WeRe-set is indeed a more powerful invariant than the (unweighted) resolution set, as the added probabilities can distinguish these two pseudoknots.

\begin{figure}[th]
\centerline{\includegraphics[width=1in]{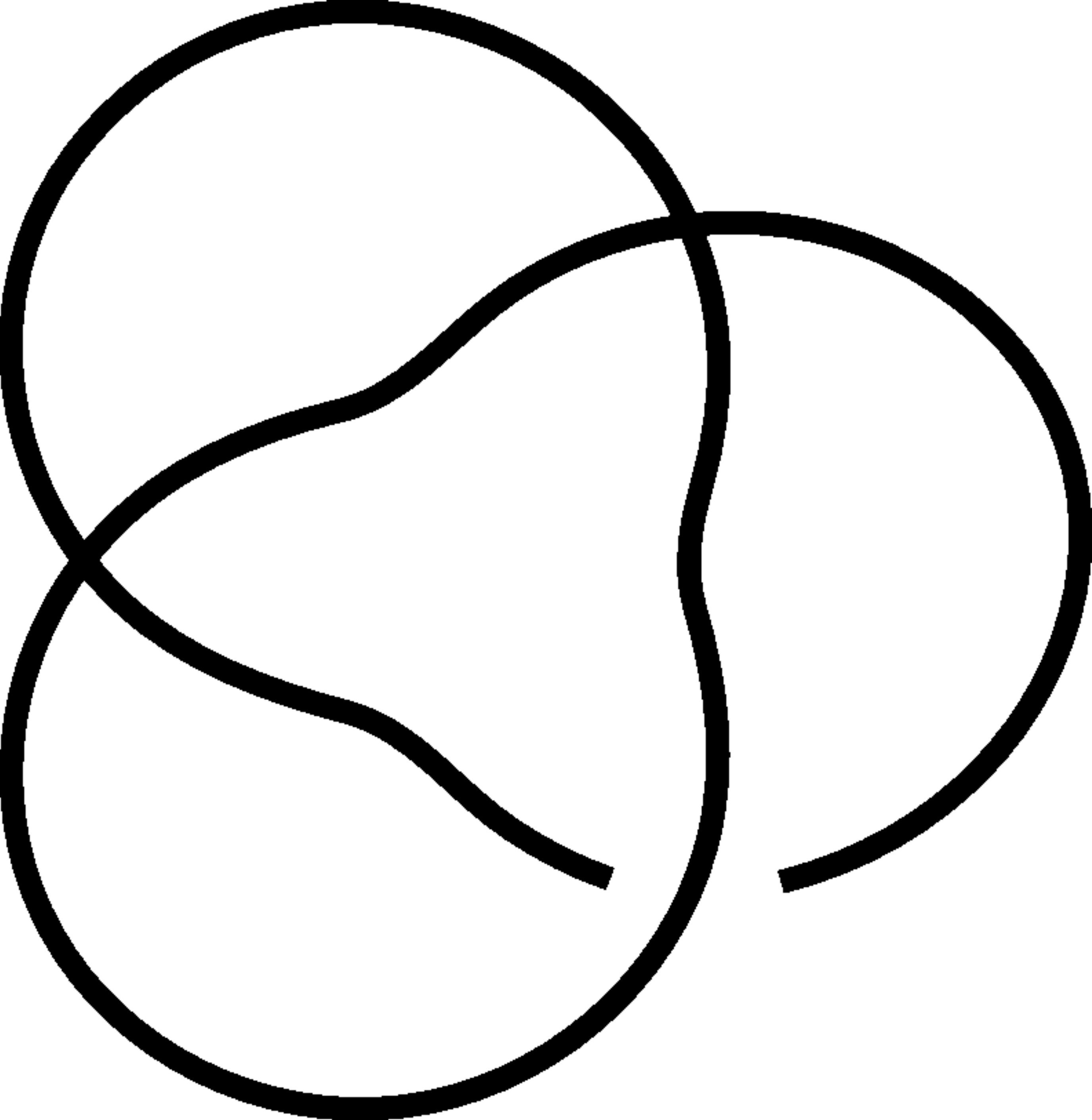}} \vspace*{8pt}
\caption{A pseudoknot, $T$, that produces the unknot and the trefoil. \label{part_tref}}
\end{figure}

\begin{theorem}
The WeRe-set is an invariant of pseudoknots.
\end{theorem}

\begin{proof} It suffices to show that the WeRe-set of a pseudoknot is unchanged by the pseudo-Reidemeister moves.  We will in fact show that all moves except the PR1 move preserve the resolution multiset (i.e. the set of knots obtained by resolving precrossings in all possible ways, with multiplicity). Therefore, they preserve the WeRe-set.

First, all classical Reidemeister moves clearly preserve the resolution multiset since they involve no precrossings.

The two PR3 moves behave much like the classical moves. Regardless of which resolution is chosen for the single precrossing, a classical R3 move is possible on the resolution, so the knot type is unchanged. Thus the WeRe-set is also unchanged.

For PR2, we note that, both before and after the move, there are two possibilities for the precrossing resolution of the pseudoknot. In each case, there is precisely one precrossing resolution that yields alternating crossings while the other yields a local diagram that R2 may be applied to. The alternating resolution that is obtained before the move is identical to the one obtained after the move. Thus, the knot types for this choice of resolution before and after the move are identical. The non-alternating choice of resolution also produces the same knot both before and after the PR2 move because a simplifying R2 move is possible in each case. It follows that the resolution multiset is unaffected by the PR2 move, and thus so is the WeRe-set.

Finally, we consider the PR1 move. We note that either resolution of the precrossing can be removed with a simplifying classical R1 move, so this move does not change the knot type of any resolution. Since we have two choices for the precrossing in PR1, the multiplicity of every resolution is increased by a factor of two after the move that adds a precrossing is performed. Since doubling the multiplicity of each knot in the resolution multiset does not affect the ratios of the resolutions, the WeRe-set is unchanged.
\end{proof}

%%%%%%%%%%%%Knot Families%%%%%%%%%%%%%
\section{Pseudoknot Families and Conway Notation}\label{families}

To help us understand the relationship between pseudoknots and their WeRe-sets, we compute the WeRe-sets for the shadows of various families of knots. Note that whenever we refer to such a shadow, we are considering the pseudoknot associated to a shadow of the standard projection of the knot.

\subsection{Torus Pseudoknots}

We begin by considering torus knots, with a focus on $(2,p)$-torus knots. A $(2,p)$-torus knot is particularly straightforward to analyze since it is the closure of a $2$-braid (i.e. a braid with two strands) that has an odd number of crossings. (Recall that the closure of a $2$-braid with an even number of crossings produces a link.) To determine the resolutions of a $(2,p)$-torus shadow, it suffices to consider the shadow of its corresponding $2$-braid.

%%%%Include diagram of a (2,p)-torus knot and its corresponding braid.%%%%

It is well known that braids can be represented by elements of a group called the {\em braid group}~\cite{birman}. For example, a $2$-braid can be represented by a word in the generators $\sigma_1$ and $\sigma_1^{-1}$, where $\sigma_1$ represents a negative crossing between the two strands and $\sigma_1^{-1}$ represents a positive crossing. (Note that the fact that these generators are inverses follows from the R2 move.) All resolutions of the shadow of a $2$-braid, therefore, correspond to powers of $\sigma_1$. This leads us to the following result.

\begin{lemma} Suppose $B$ is the shadow of a $2$-braid with $n$ crossings. Then there are ${n\choose{k}}$ ways to resolve the precrossings to obtain the braid $\sigma_1^{n-2k}$. Moreover, if $B'$ is a pseudodiagram of a $2$-braid with $n$ precrossings and the classical crossings contribute a total of $\sigma_1^l$ to the braid word for $B'$, then there are $n \choose k$ ways to resolve $B'$ to get a braid with reduced word $\sigma_1^{l+n-2k}$.
\label{braid}\end{lemma}

\begin{proof} If we choose $n-k$ of the $n$ crossings to be negative (corresponding to $\sigma_1$) and $k$ to be positive (corresponding to $\sigma_1^{-1}$), then the reduced braid word corresponding to the resulting braid is $\sigma_1^{n-2k}$. Furthermore, there are ${n\choose{k}}$ ways to choose the $k$ positive crossings. The second statement follows as an immediate corollary.
\end{proof}

\begin{theorem} Every resolution of the shadow of a $(2,p)$-torus knot is a $(2,p-2k)$-torus knot, with $0\leq k \leq p$. Moreover, there are ${p\choose{k}}$ ways to obtain a $(2,p-2k)$ torus knot from the $(2,p)$-torus shadow. In particular, since the unknot is its own mirror image (as it can be represented both as a $(2,1)$- and $(2,-1)$-torus shadow), there are $2{p\choose {\lfloor p/2\rfloor}}$ ways to get the unknot.
\label{2ptorus}\end{theorem}

\begin{proof} A $(2,p)$-torus knot is the closure of the 2-braid with $p$ crossings corresponding to the braid word $\sigma_1^p$. By the previous lemma, there are $p\choose k$ ways to get the braid $\sigma_1^{p-2k}$ from the shadow of the 2-braid with $p$ crossings, and the closures of these braids are precisely the $(2,p-2k)$-torus knots. Clearly, no other knot types are obtainable from the standard projection, since every resolution is the closure of a $2$-braid.
\end{proof}

\begin{corollary} The knotting number of the shadow of a $(2,p)$-torus knot is equal to $\frac{p+3}{2}$.
\end{corollary}

\begin{proof} Any crossing can be removed by a suitable resolution of an adjacent precrossing. So if we have chosen $k$ crossings (say, positively) and have $p-k$ left to choose, the difference must be at least $3$ to ensure that a nontrivial knot is produced. If we let $k=\frac{p+3}{2}$, then $$k-(p-k) =\frac{p+3}{2}-\frac{p-3}{2}=3.$$ This is the least number of crossing choices we can make to be sure we have a knot.
\end{proof}

While the WeRe-sets of $(2,p)$-torus knot shadows are straightforward to describe, other kinds of torus knot shadows present more of a challenge. One interesting feature of the $(2,p)$-torus knots is that every resolution is itself a $(2,p)$-torus knot for some $p$. For the $(3,4)$-torus shadow, however, a similar result fails to hold.

\begin{example}[The (3,4)-torus knot.]
The WeRe-set of the shadow of the (3,4)-torus knot  is

\footnotesize

$$\{(0_1,{88\over 2^8}),(3_1,{72\over 2^8}),(4_1,{4\over 2^8}),(5_1,{16\over 2^8}),(5_2,{32\over 2^8}), (6_3,{16\over 2^8}),(8_{18},{2\over 2^8}),(8_{19},{2\over 2^8}),(8_{20},{16\over 2^8}),(3_1\# 3_1,{8\over 2^8})\}.$$

\normalsize

\noindent We note that it contains knots which are not torus knots.
\end{example}

The $(3,7)$-torus knot is also of particular interest.

 \begin{example}[The (3,7)-torus knot.]
The probability of obtaining an unknot from a knot shadow is higher than the probability of obtaining any other knot type for shadows with up to 12 crossings. An exception to this rule is the shadow of the $(3,7)$-torus knot with the WeRe-set $$\{(0_1,{2688 \over 2^{14}}),(3_1,{2884 \over 2^{14}}),\ldots \}.$$
\end{example}

\begin{figure}[th]
\centerline{\includegraphics[width=3.4in]{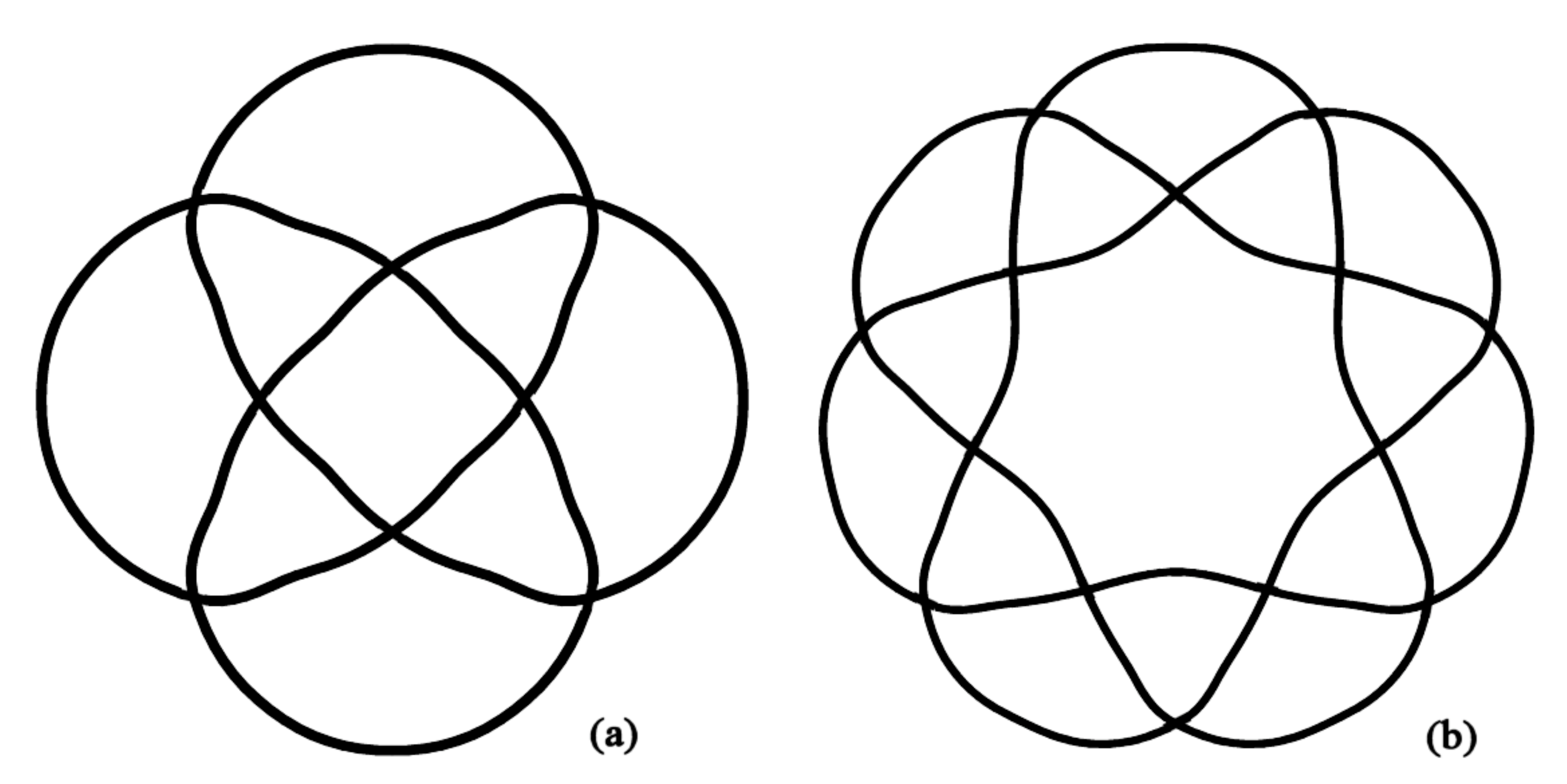}} \vspace*{8pt}
\caption{The shadows (a) of the $(3,4)$-torus knot; (b) of the $(3,7)$-torus knot. \label{f0}}
\end{figure}

\subsection{Rational Pseudoknots}

We turn our attention from pseudoknots related to torus knots to other interesting knot families. We consider the concrete example of shadows of twist knots (as in Figure~\ref{twist}) then discuss rational pseudoknots more generally. For more on rational and twist knots, see~\cite{adams}.

\begin{figure}[th]
\centerline{\includegraphics[width=1.3in]{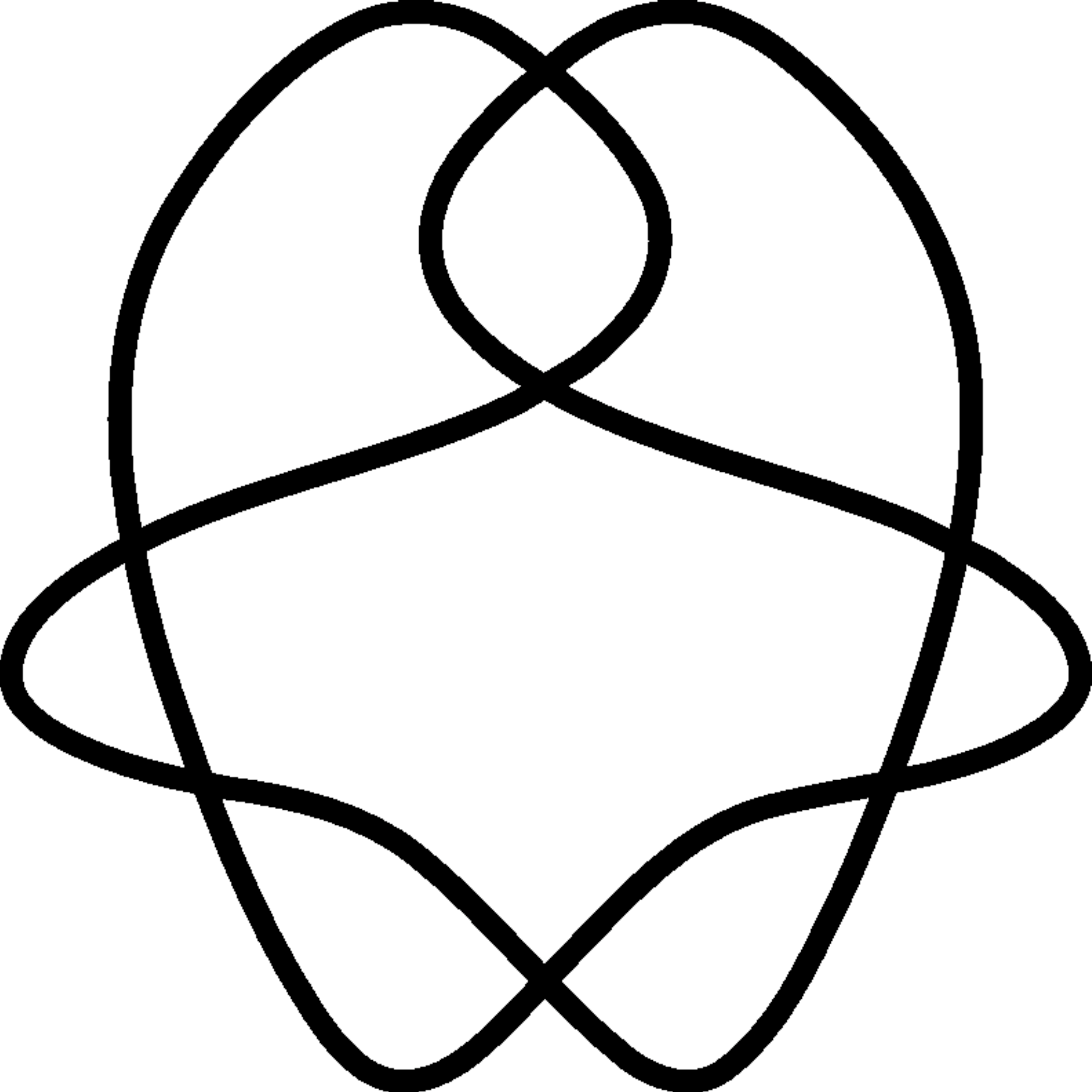}} \vspace*{8pt}
\caption{The shadow of a twist knot. \label{twist}}
\end{figure}

\begin{theorem} Every resolution of a twist knot shadow with $n$ crossings is either a twist knot with $n$ or fewer crossings or the unknot. The number of ways to obtain a twist knot with $n-2k$ crossings by resolving the shadow of a twist knot with $n$ crossings is $2{{n-2} \choose{k}}$. The number of ways to obtain the unknot is $2^{n-1}+ 2{{n-2} \choose{\lfloor n-2/2\rfloor }}$.
\end{theorem}

\begin{proof}
A twist knot is made up of a clasp and a twist, and each of these tangles can be viewed as braids. The clasp of a twist knot shadow can be resolved to alternate in either of two ways (corresponding to the braids $\sigma_1^2$ and $\sigma_1^{-2}$). It can also be resolved so that it is not alternating, in which case our resolution is the unknot.

By Lemma~\ref{braid}, there are $n \choose k$ ways for the twist tangle to be the braid $\sigma_1^{n-2k}$. If we have the braid $\sigma_1^{n-2k}$ and the clasp is alternating, we will either get a twist knot with $n-2k$ crossings or one with $n-2k-1$ crossings, depending on which way the clasp alternates. Since we can get both the twist knot with positive crossings and the one with negative crossings, we multiply by 2 to get the total number of possible resolutions that will give us a particular twist knot.
\end{proof}

%%%Make a diagram of a rational knot shadow to replace this with%%%

\begin{figure}[th]
\centerline{\includegraphics[width=1.3in]{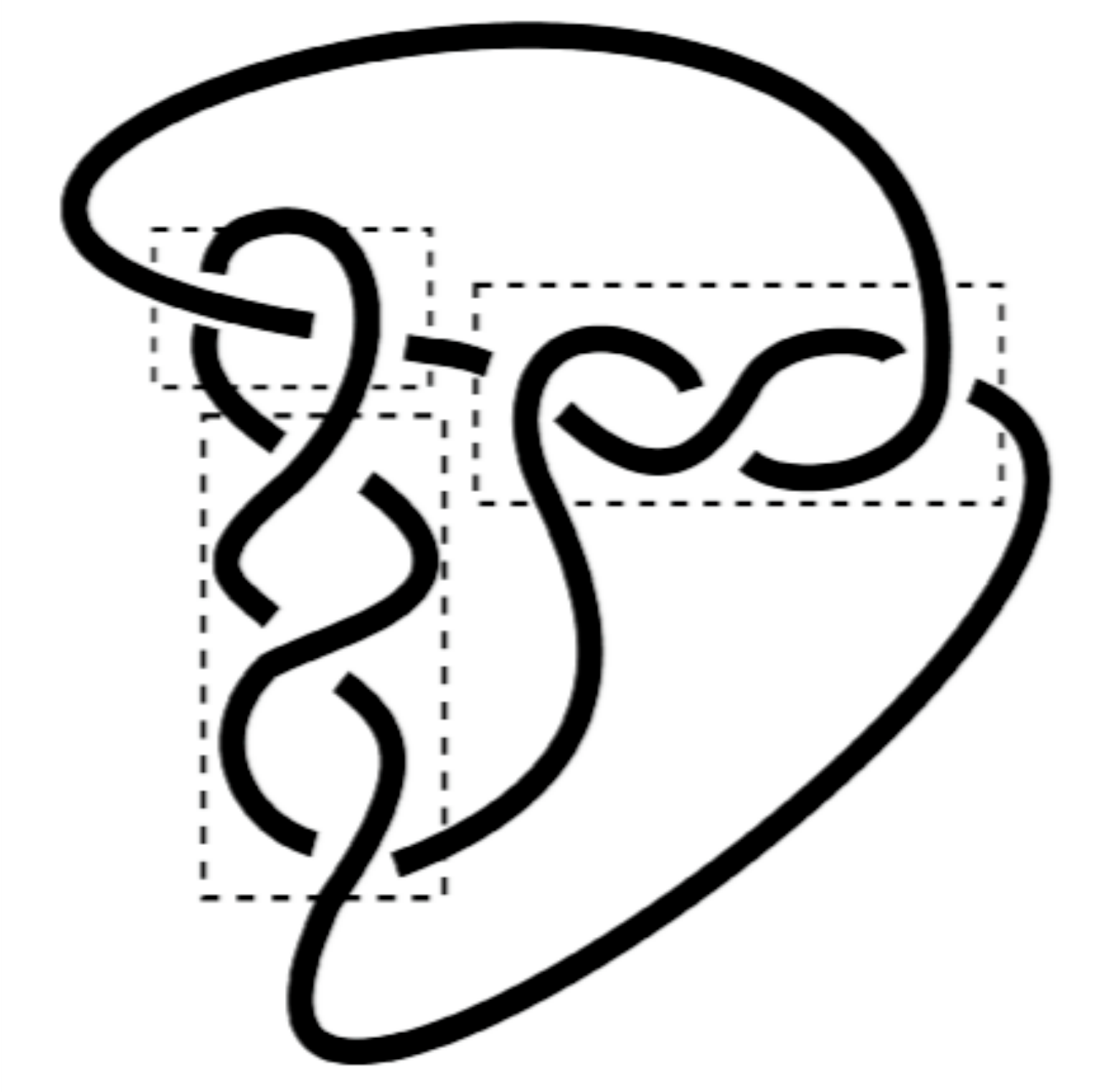}} \vspace*{8pt}
\caption{A rational knot. \label{rational}}
\end{figure}

Twist knots are a subfamily of the larger family of {\em rational knots}~\cite{conway,lousophia}. We can generalize our strategy of computing WeRe-sets for twist knot shadows to strategies for computing rational knot shadow WeRe-sets. Recall that a rational knot can be viewed as the closure of a rational tangle, that is, a tangle made up of a sequence of horizontal and vertical twists. See Figure~\ref{rational} for an example of a rational knot, and refer to~\cite{adams} for a detailed discussion of rational tangle construction.

The sequence of twists $a_1\, a_2\, ...\, a_n$ (meaning the first twist has $a_1$ crossings, the second has $a_2$ crossings, etc.) that creates a rational knot is key for classifying such knots. (Note that twist knots are associated to sequences of the form $a_1\, 2$.) From this sequence, we can form the following continued fraction, which is instrumental in determining information about knot type.

\begin{equation*} a_n + \frac{1}{ a_{n-1} + \frac{1}{ ...+\frac{1}{a_1}}}
\end{equation*}
If this is equal to $p/q$ for some relatively prime $p,q \in \mathbb{Z}$ with $p$ odd, then it is a knot. Moreover, a theorem by Schubert from~\cite{schubert} tells us exactly when two such knots will be equal.

\begin{theorem}[Schubert~\cite{schubert}]
Suppose that rational tangles with fractions $p/q$ and $p'/q'$ are given, where $p$ and $q$ are relatively prime as are $p'$ and $q'$. (The numerators are always assumed to be positive, and the denominators may be negative.) If $K(p/q)$ and $K(p'/q')$ denote the corresponding rational knots obtained by taking numerator closures of these tangles, then $K(p/q)$ and $K(p'/q')$ are topologically equivalent if and only if
\\1. $p=p'$ and
\\2. either $q \equiv q' \mod{p}$ or $qq' \equiv1 \mod{p}$.
\end{theorem}

Schubert's theorem, along with the following theorem, help us determine the WeRe-sets for shadows of rational knots.

\begin{theorem} Suppose $S$ is the (canonical) shadow of the rational knot with sequence $a_1\, a_2\, ...\, a_n$. Then the number of resolutions of $S$ that yield the rational knot $(a_1-2k_1)\, (a_2-2k_2)\, ...\, (a_n-2k_n)$ is $$\prod_{i=1}^n{a_i \choose k_i}.$$
\end{theorem}

\begin{proof} We note that each element of the sequence $a_1\, a_2\, ...\, a_n$ represents the shadow of a tangle sub-diagram that is a $2$-braid. The total number of ways to get the sequence $(a_1-2k_1)\, (a_2-2k_2)\, ...\, (a_n-2k_n)$, then, is the product of the number of ways to get each element in the sequence, as determined in Lemma~\ref{braid}.
\end{proof}

\subsection{Conway Notation for Pseudoknots}
Thus far, we have only considered rational knot shadows and rational knot diagrams. If we are to consider more generally diagrams of rational pseudoknots, or Conway notation for all pseudoknots, we need to expand our classical notational conventions.

\begin{figure}[th]
\centerline{\includegraphics[width=3in]{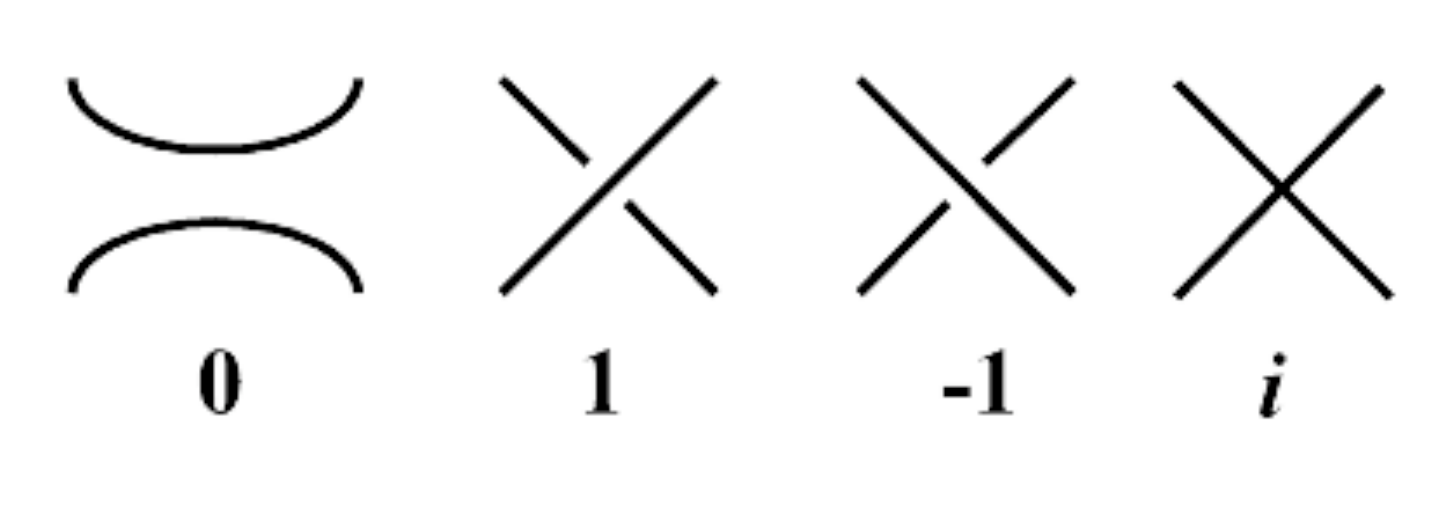}} \vspace*{8pt}
\caption{The elementary tangles. \label{f1}}
\end{figure}

The Conway symbols of knots with at most $10$ crossings and
links with at most $9$ crossings are given in the Appendix of Rolfsen's book \cite{2}. Expanding the Conway notation for classical knots and links \cite{conway,2,caudron,4},
we define Conway notation for pseudoknots and pseudolinks by adding to the list of the elementary tangles $0$, $1$, $-1$ the elementary tangle $i$, which denotes a precrossing. See Figure~\ref{f1}. In our extended Conway notation for pseudoknots and pseudolinks, we have the following conventions:
\begin{itemize}
  \item A sequence of $n$ precrossings (that is, a precrossing $n$-twist) is denoted by $i,\ldots ,i$, or simply $i^n$.
  \item A sequence of $n$ classical crossings (that is, a positive $n$-twist) is denoted by $1,\ldots ,1$, or $1^n$.
  \item A sequence of $n$ classical negative crossings (negative $n$-twist) is denoted by $-1,\ldots ,-1$ , or $(-1)^n.$
\end{itemize}

In this way, we also obtain mixed pseudotwists, e.g., a pseudotwist $(i,i,1,1,1,-1,-1)$, which can be shortly written as $(i^2,1^3,(-1)^2)$. A pseudotwist will be called {\em reduced} if it has the minimal number of crossings among all pseudotwists equivalent to it. Every reduced pseudotwist will be positive or negative, i.e., it will not contain crossings of different signs. Since crossings in a pseudotwist commute, every reduced twist can be ordered and represented in the the form $(i^p,1^q)$ or $i^p,(-1)^q$  where $p,q\ge 0$. Minimal diagrams of pseudoknots will contain only ordered reduced twists.

\begin{figure}[th]
\centerline{\includegraphics[width=3in]{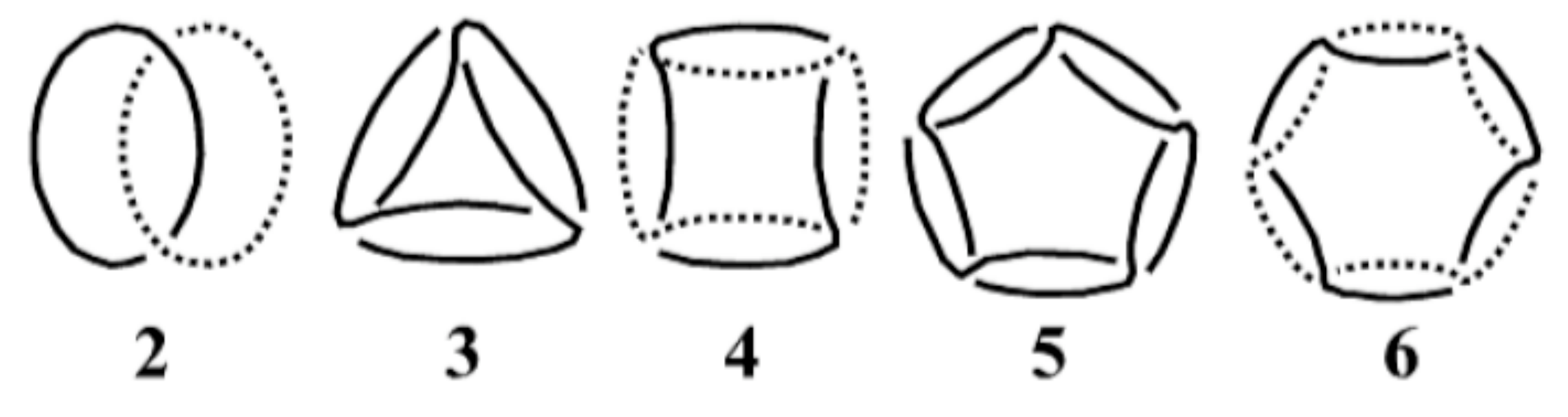}} \vspace*{8pt}
\caption{Hopf link $2_1^2=2$ and its family $p$ ($p=2,3,\ldots $).
\label{f9}}
\end{figure}

For an example which illustrates how the notation is used, we consider the simplest family $2_1^2$, $3_1$, $4_1^2$, $5_1$, $\ldots$ of torus knots and links which includes the Hopf link $2$, the trefoil $3$, etc. This family can be denoted in Conway notation by the general symbol $p$ ($p\ge 2$). By substituting a precrossing for a classical crossing in each family member, we obtain the family of pseudoknots and pseudolinks $(i,1)$, $(i,1,1)=(i,1^2)$, $(i,1,1,1)=(i,1^3)$, $(i,1,1,1,1)=(i,1^4)$, $\ldots$, and in general $(i,1^{p-1})$ ($p\ge 2$).

%%%%%%%%%%%Crossing Number%%%%%%%%%%%
\section{Crossing Number of a Pseudoknot}\label{crossnum}

Just as with classical knots, we have a natural notion of crossing number for pseudoknots, as follows.

\begin{definition} The {\em crossing number}, $cr(K)$, of a pseudoknot $K$ is the minimum number of total crossings (both classical and precrossings) of any projection of that pseudoknot.\end{definition}

Given this definition, one natural question that arises is the following. Is the crossing number of a pseudoknot equal to the maximum crossing number of its resolutions? Equivalently, if it is possible to reduce any resolution of a pseudoknot, is it necessarily possible to reduce the pseudoknot?\\

Using the WeRe-set invariant, we are able to answer this question negatively. Indeed, there are pseudoknots with crossing number strictly greater than the maximum of the crossing numbers of knots in the  resolution set. Consider the pseudodiagram given in Conway notation by $(3)\,(i)\,(-2)$. The pseudoknot $K$ given by this diagram has WeRe-set $\{(5_1,\frac{1}{2}),(5_2,\frac{1}{2})\}$, which is different from the WeRe-set of any pseudoknot with crossing number five. We can conclude that the crossing number of $K$ is six, while the maximum of the crossing numbers of knots in the WeRe-set is five.

On the other hand, there is a nice class of examples for which our question can be answered positively. First, some terminology is needed.

\begin{definition} A crossing in a pseudodiagram is called {\em nugatory} if there exists a circle in the plane enclosing part of the knot which intersects the knot at the crossing and nowhere else. \end{definition}

\begin{figure}[th!]
\centerline{\includegraphics[width=3in]{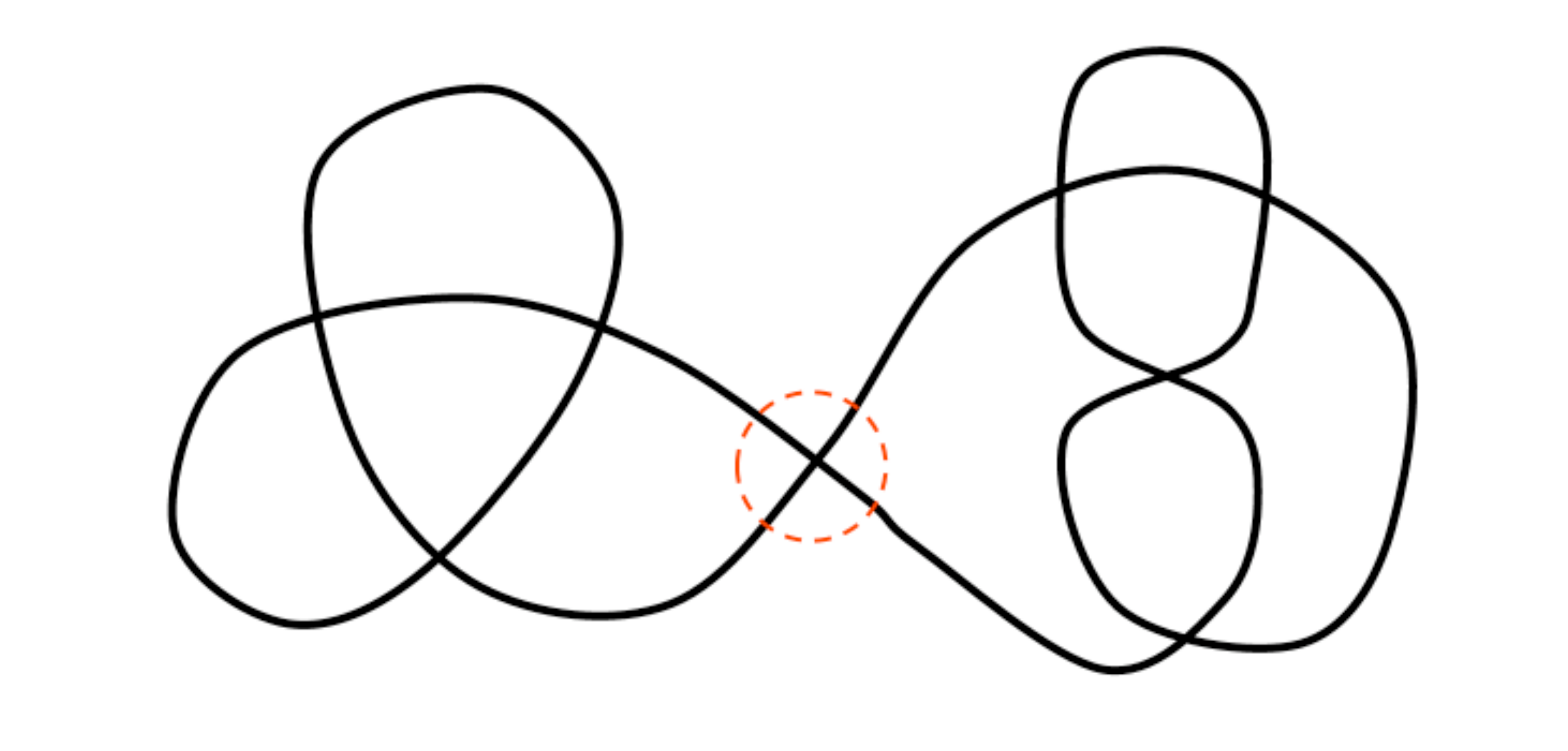}} \vspace*{8pt}
\caption{A nugatory crossing.
\label{nug}}
\end{figure}

\begin{definition} A pseudodiagram is {\em potentially alternating} if it is has no nugatory crossings and if it is possible to resolve its precrossings so that the resulting knot diagram is alternating.\end{definition}

We note that any shadow is potentially alternating, while our example above, $(3)\,(i)\,(-2)$, is not potentially alternating. Returning to our question, we find the following result.

\begin{proposition}
Suppose a pseudoknot $K$ has a potentially alternating diagram $D_K$. Then the crossing number of $K$ is realized in $D_K$. Furthermore, the crossing number of $K$ is equal to the maximum crossing number of the knots in the resolution set of $K$.
\end{proposition}

\begin{proof}
The result follows immediately from the fact that $K$ has a reduced alternating knot $K_a$ in its resolution set with $cr(K_a)=cr(D_K)$.
\end{proof}

\begin{remark} In our example $(3)\,(i)\,(-2)$, we saw that if crossing $i$ is resolved one way, the knot $5_1$ results. If the crossing is resolved the other way, we get the knot $5_2$. An interesting question to consider is whether there is a similar example of a pseudodiagram with a single non-nugatory precrossing such that both resolutions produce the same knot rather than two different knots (so the WeRe-set has a single element). In fact, the rational knot $(1,1,1) (i) (-1,-1,-1)$ is such an example for {\em unoriented} pseudoknots. Both resolutions of the precrossing yield the knot $6_1$.

The precrossing in a pseudodiagram with this curious property are called {\em cosmetic} or (non-nugatory) {\em degenerate} crossings. Nugatory crossings are also considered to be degenerate, but there are not yet any known oriented knots containing non-nugatory degenerate crossings (i.e. cosmetic crossings). This problem was first posed by X.-S. Lin and appears as Problem 1.58 in Rob Kirby's Problem List.\end{remark}

%%%%%%%%%%%%Homotopy of Pseudoknots%%%%%%%%%%%%%
\section{Homotopy of Pseudoknots}\label{hom}

In this section, we consider more broadly the relationship between pseudoknots that is generated by crossing changes.

\begin{definition}
Two pseudoknots are called {\it homotopic}, or {\it crossing-change equivalent}, if their pseudodiagrams can be related by sequences of crossing changes and pseudo-Reidemeister moves.
\end{definition}

In the case of classical knots, the notion of homotopy is uninteresting. Indeed, every classical knot is homotopic to the unknot. The notion of homotopy for pseudoknots, however, is nontrivial. While there exist homotopically nontrivial pseudoknots, the following result gives a large class of pseudoknots which remain homotopic to the unknot.

\begin{theorem}
Every pseudoknot that can be represented by a pseudodiagram with a single precrossing is homotopic to the unknot. %Every pseudoknot with $k$ precrossings ($k\ge 2$) is crossing-change equivalent to the unknot, or to a pseudoknot containing pseudotwists with at most one 1 crossing ({\bf prove!}).
\end{theorem}

\begin{proof} First, recall that a knot diagram is {\it based} if a base point (different from the crossing points) is specified on the diagram, and {\it oriented} if an orientation is assigned to it. Let $K$ be a knot and $\widetilde{K}$ be a based oriented diagram of $K$. The {\it descending diagram} of $\widetilde{K}$, denoted by $d(\widetilde{K})$, is obtained as follows: beginning at the basepoint of $\widetilde{K}$ and proceeding in the direction specified by the orientation, change the crossings as necessary so that each crossing is first encountered as an over-crossing. Note that $d(\widetilde{K})$ is the diagram of a trivial knot. Every knot diagram can be turned by a finite number of crossing changes into a descending diagram, so the unknotting number of every knot is finite.

We apply a similar idea to pseudoknots that can be represented by a pseudodiagram with a single precrossing. The precrossing divides the pseudoknot in two subcomponents. Choose the basepoint such that between it and the precrossing there are no crossings. In the first step, beginning at the basepoint and proceeding in the direction specified by the orientation, change the crossings as necessary so that the first subcomponent is always over the second. Then, pseudo-Reidemeister moves may be performed so that there are no crossings involving both subcomponents. In effect, the precrossing becomes nugatory. In the next step, beginning at the basepoint and proceeding in the direction specified by the orientation, change the crossings as necessary to make each of the subcomponents descending. The obtained diagram is a pseudodiagram of the unknot with a single precrossing. Hence, every pseudoknot with a single precrossing is homotopic to the unknot, so the unknotting number of pseudoknots with a single precrossing is a finite invariant. %({\bf prove the second part of the theorem!})
\end{proof}

\begin{remark} We note that every pseudoknot with $k=2$ precrossings is homotopic to the unknot or to the pseudoknot $3_1.2=(i^2,1)$. Moreover, every pseudoknot with $k=3$ precrossings is homotopic to the unknot, the pseudoknot $3_1.2=(i^2,1)$, or one of the pseudoknots $3_1.1=(i^3)$ or $4_1.2=(i^2)\,(i,1)$. Every pseudoknot with $k=4$ precrossings is crossing-change equivalent to one of the preceding pseudoknots, or to one of the following pseudoknots:  $(i^2)\,(i^2)$, $(i^4,1)$, $(i^3)\,(i,1)$, $(i^2,1)\,(i)\,(i,1)$, $(i,1)\,(i)\,(i)\,(i,1)$, $(i^2,1)\#(i^2,1)$, {\it etc}. \end{remark}

Of course, these results are only interesting if there exist homotopically nontrivial pseudoknots. For instance, how do we know that $3_1.2$ is homotopically nontrivial? This fact is a consequence of the following theorem, which uses the notion of a Gauss diagram corresponding to a pseudodiagram. For more on this extension of the theory of Gauss diagrams, see~\cite{SMALL}.

\begin{theorem}
Let $K$ be a pseudodiagram with (at least) two precrossings whose corresponding chords intersect in the Gauss diagram for $K$. Then any pseudodiagram $K'$ that is homotopic to $K$ must be nontrivial. In particular, the Gauss diagram for $K'$ must contain intersecting chords corresponding to two precrossings. 
\end{theorem}

\begin{proof} This result follows from Lemma 3.6 in~\cite{SMALL} which states that any pseudodiagram with a Gauss diagram of this form must have a nontrivial resolution. Moreover, we consider the effect that PR moves and the crossing change move have on the Gauss diagram of a pseudodiagram. If a Gauss diagram has two intersecting {\em prechords} (i.e. chords corresponding to precrossings), then any equivalent Gauss diagram has two intersecting prechords, as no Gauss diagram version of a PR or crossing change move can have the effect of uncrossing these chords.
\end{proof}

\begin{remark} For those unfamiliar with the theory of Gauss diagrams, we note that the requirement that two prechords cross in the Gauss diagram can be reformulated as follows. Given a pseudodiagram with precrossings $a$ and $b$, if one can travel along one half of the knot from $a$ back to itself, encountering $b$ exactly once along the way, then the prechords for $a$ and $b$ intersect in the Gauss diagram.\end{remark}

In every homotopy class, a pseudoknot realizing the minimum crossing number will be used as the representative of this class and denoted by $K_0$. In the cases where an appropriate $K_0$ can be determined, the following notion is a useful analog of the unknotting number.

\begin{definition}
The $K_0$ crossing-change number $u_{c}$ of the pseudoknot $K$ belonging to the homotopy class of $K_0$ is the minimum number of crossing changes taken over all pseudodiagrams of $K$ needed to obtain the pseudoknot $K_0$.
\end{definition}

In the case of pseudoknots that are homotopic to the unknot, the $K_0$ crossing-change number will be simply called {\it pseudounknotting number}. Just as with the unknotting number for classical knots, pseudounknotting number cannot be computed from minimal diagrams. To show this, we can use the pseudoknot equivalent of the Nakanishi-Bleiler example: the fixed diagram of  $(i,1^4)\,(1)\,(1^4)$ requires at least three crossing changes to unknot, and its unknotting number is at most 2. This value can be obtained if in $(i,1^4)\,(1)\,(1^4)$, we make one crossing change and obtain $(i,1^4)\,(-1)\,(1^4)=(i,1^2)\,(1)\,(1^2)$. Then by the crossing change $(i,1^2)\,(-1)\,(1^2)$, we obtain the unknot.

We can formulate the pseudoknot equivalent of the Bernhard-Jablan Conjecture \cite{4} as follows. Starting from a minimal diagram of $K$ with $n$ classical crossings, we make all single crossing changes to obtain $n$ new pseudodiagrams. We then reduce the pseudodiagrams obtained to their minimal diagrams using pseudo-Reidemeister moves. Then, we continue with this recursive process until $K_0$ is obtained. The conjecture states that the $K_0$ crossing-change number $u_{c}$ of $K$ will be equal to the number of steps in this recursive crossing-change process, which is denoted $u_{BJ}$.

We consider a question whose solution relies on this conjecture. If $K_1$ and $K_2$ are the resolutions of a pseudoknot $K$ with a single precrossing, it is clear that $u_c(K)\ge max(u(K_1),u(K_2))$, where $u$ denotes the classical unknotting number. Is there a pseudoknot of this form that realizes the strong inequality $u_c(K)> max(u(K_1),u(K_2))$? How large can the difference $u_{BJ}(K)-max(u(K_1),u(K_2))$ be?

Because it is difficult to find exact unknotting numbers for pseudoknots, we restrict our consideration to the unknotting numbers $u_{BJ}$. We find that the difference $u_{BJ}(K)-max(u(K_1),u(K_2))$ can be arbitrarily large. An example illustrating this phenomenon is the family of pseudoknots $(1^p)\,(1^2)\,(1)\,(i,1),$ ($p\ge 2$). For this family, the $BJ$-unknotting number is $\lfloor {{p+1}\over 2}\rfloor +1$ while the unknotting numbers of both resolutions are 1.

%%%%%%%%%%%%Signed and Oriented WeRe-sets%%%%%%%%%%%%%
\section{Mirror Images and Chirality}\label{chiral}
\subsection{Signed WeRe-Sets}

In the previous sections, we used the WeRe-set as our primary invariant for distinguishing pseudoknots. In fact, there is a stronger version of the WeRe-set that remains invariant. The {\em signed weighted resolution set} of a pseudoknot is a WeRe-set where distinctions are made between resolutions and their mirror images. While Theorem~\ref{2ptorus} is in fact a result about signed WeRe-sets, we have not often made the distinction between a knots and their mirror images. The following example illustrates how this distinction can be useful.

Consider the standard trefoil shadow, $(i^3)$. The WeRe-set of this pseudoknot is $\{(0_1,{6\over 2^3}),(3_1,{2\over 2^3})\}$ (where $3_1$ may denote both the right- and the left-handed trefoils). On the other hand, the trefoil with two precrossings and a positive crossing, $(i^2,1)$, has the same WeRe-set, despite the fact that the  left-handed trefoil, $-3_1$, is not a possible resolution of $(i^2,1)$ but is a possible resolution of $(i^3)$. Indeed, the signed WeRe-set of $(i^2,1)$ is $\{(0_1,{6\over 2^3}),(+3_1,{2\over 2^3})\}$, where $+3_1$ emphasizes that the resolution is the right-handed trefoil. Of course, $-3_1$ is in the (signed) WeRe set, $\{(0_1,{6\over 2^3}),(3_1,{1\over 2^3}),(-3_1,{1\over 2^3})\}$, of $(i^3)$. So we see that the signed WeRe-set can distinguish these two pseudoknots.

There are other interesting examples of pseudoknots with the same WeRe-sets but different signed WeRe-sets. For instance, the pseudoknots $(1,1,1)\,(i)\,(-1,-1,-1)$ and $3\,2$ have the same WeRe-set $\{(6_1,1)\}$. However, their signed WeRe-sets are $\{(6_1,{1\over 2}),(-6_1,{1\over 2})\}$ and $\{(6_1,1)\}$ respectively.

%%%%%%%%%%%%Pseudoknot Chirality%%%%%%%%%%%%%
\subsection{Pseudoknot Chirality}

When we begin to make distinctions between knots and their mirror images, it is natural to ask questions regarding which knots are equivalent to their mirror images, i.e. which knots are {\em amphicheiral}. The mirror image $\overline{D}$ of a pseudodiagram $D$ is the diagram obtained from $D$ by making all positive crossings negative and vice versa, leaving all precrossings untouched. Two pseudoknots are mirror images of one another if thy have pseudodiagrams that are mirror images. Given this notion of mirror image for pseudoknots, we introduce the related notion of amphicheirality for pseudoknots.

\begin{definition}
A pseudoknot $K$ is called {\it amphicheiral} if $K$ and its mirror image $\overline{K}$ are ambient isotopic.
\end{definition}

In order to recognize amphicheiral pseudoknots we can use signed WeRe-sets, where to an amphicheiral pseudoknot $K$ and its mirror image $\overline{K}$ correspond the same signed WeRe-sets.

In order to construct an infinite number of amphicheiral rational pseudoknots with one or two precrossings, we can use the following two theorems for classical knots:

\begin{theorem}[Caudron~\cite{caudron}, Siebenmann~\cite{5}]
A rational knot is amphicheiral {\it iff} its Conway symbol is
mirror-symmetric and has an even number of crossings.
\end{theorem}

\begin{theorem}[Kanenobu-Murakami~\cite{6}]
Every rational unknotting number 1 knot can be expressed by
one of the following Conway symbols
$$c_0\,c_1\,\ldots \,c_{r-1}\,c_r\,1\,1\,(c_r-1)\,c_{r-1}\,\ldots
\,c_1$$
$$c_0\,c_1\,\ldots \,c_{r-1}\,(c_r-1)\,1\,1\,c_r\,c_{r-1}\,\ldots
\,c_1,$$ \noindent where $c_i\ge 0$ for $i=0,\ldots ,r$ and
$c_r\ge 2$.
\end{theorem}

The knot family $p\,1\,1\,p$ ($p\ge 2$), which includes knots $6_3$, $8_9$, $10_{17}$,$\ldots $, satisfies the criteria of both theorems and consists of amphicheiral rational knots with unknotting number 1. If we substitute a crossing $1$ by a precrossing in any knot in this family, we obtain the family of pseudoknots $(1^p)\,(i)\,(1)\,(1^p)$ (pictured in Fig. \ref{f10}a). The resolutions of this pseudoknot are the unknot and knot $p\,1\,1\,p$, which are both amphicheiral. Thus, all pseudoknots $(1^p)\,(i)\,(1)\,(1^p)$ are amphicheiral.

From this same knot family, we derive the family of amphicheiral pseudoknots $(1^p)\,(i)\,(i)\,(1^p)$ with two precrossings. Indeed, we notice that its resolutions are the amphicheiral knots $p\,1\,1\,p$, the unknots $p\,(-1)\,1\,p$ and $p\,1\,(-1)\,p$, and the amphicheiral knot $p\,(-1)\,(-1)\,p=(p-1) 1 1 (p-1)$ (see Fig. \ref{f10}b).

\begin{figure}[th]
\centerline{\includegraphics[width=3.8in]{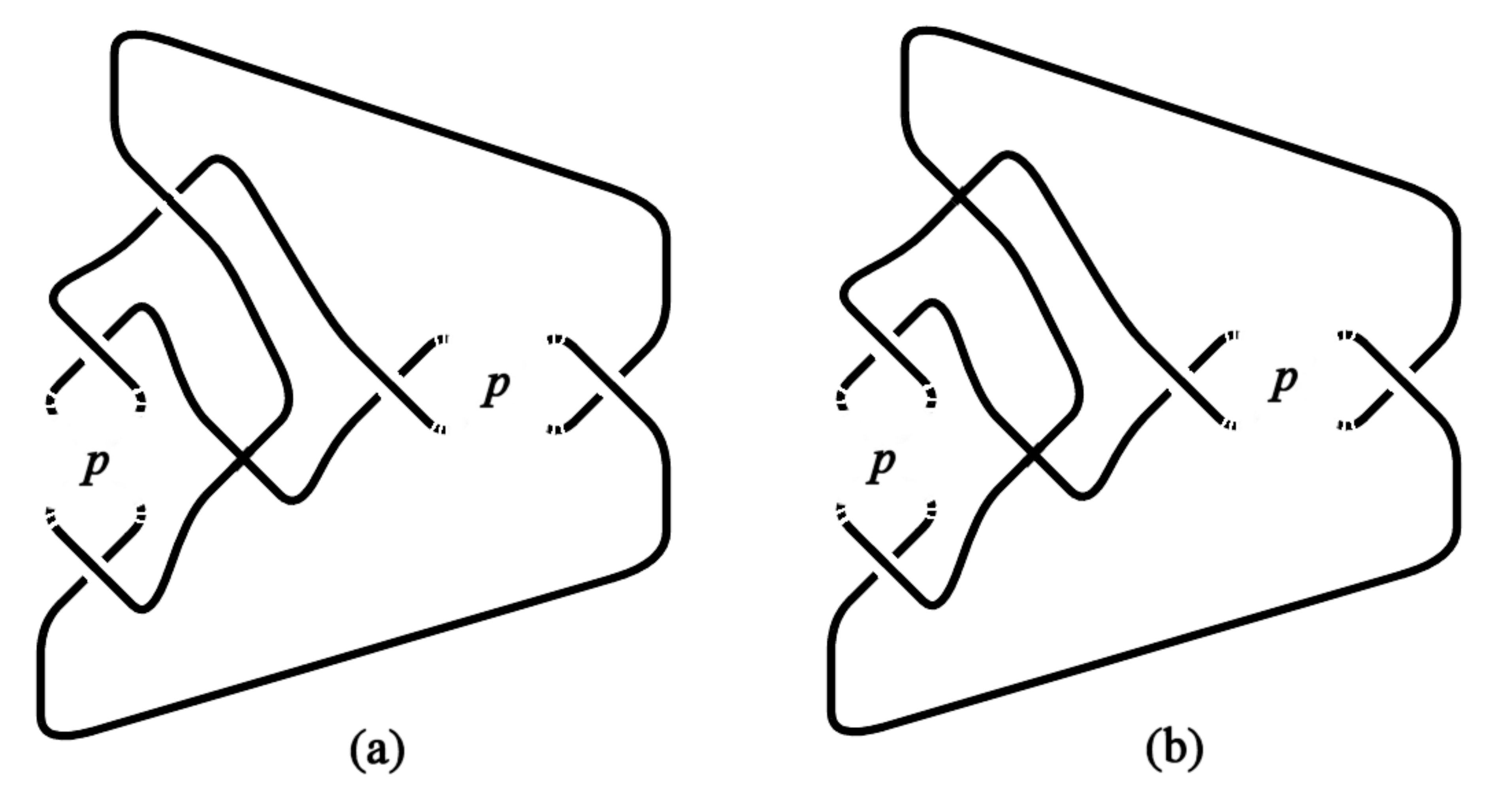}} \vspace*{8pt}
\caption{Amphicheiral pseudoknots (a) $(1^p)\,(i)\,(1)\,(1^p)$; (b) $(1^p)\,(i)\,(i)\,(1^p)$.
\label{f10}}
\end{figure}

%%%%%%%%%%%Conclusion%%%%%%%%%%%%
\section{Additional Open Questions}\label{quest}

While we have alluded to several open questions throughout our paper, we conclude with several more interesting questions about pseudoknots.

\smallskip

\noindent {\bf Question 1:} Is each pseudoknot uniquely determined by its signed WeRe-set?

\smallskip

\noindent {\bf Question 2:} Do flypes involving nugatory precrossings preserve pseudoknot type?

\smallskip

\noindent {\bf Question 3:} Earlier, we considered shadows of $2$-braids. What is an appropriate definition, in general, of pseudobraids? In particular, when are two pseudobraids equivalent? Furthermore, in classical braid theory there are Markov moves that characterize when two braids have equivalent closures. Is there an analog for pseudobraids?

%%%%%%%%%%%%Appendix%%%%%%%%%%%%
\appendix
\section{Pseudoknot Tables}
Knot tables of prime pseudoknots with at most 9 crossings are based on the computation of their WeRe-sets. The number of obtained pseudoknots is given in the following table:

\bigskip

\noindent  \begin{tabular}{|c|c|} \hline
$n$ &  \\  \hline
3 & 3 \\  \hline
4 & 5 \\  \hline
5 & 15 \\  \hline
6 & 59 \\  \hline
7 & 212 \\  \hline
8 & 1344 \\  \hline
9 & 7281 \\  \hline
\end{tabular}

\bigskip

In this paper we provide tables of pseudoknots with at most 5 crossings, given by their Conway symbols, WeRe-sets, and followed by diagrams. The remaining part of the tables can be downloaded from the address: {\tt http://www.mi.sanu.ac.rs/vismath/pseudotab.pdf}. In this file we used a concise notation for WeRe-sets. For example, instead writing the complete term with fractions, e.g., $(i,i,i), \{(0_1,{6\over 2^3}),(3_1,{2\over 2^3})\}$, we wrote just $(i,i,i), \{(0_1,6),(3_1,2)\}$, knowing that this pseudoknot has 3 precrossings, so the second entries in the term $\{(0_1,6),(3_1,2)\}$ need to be divided by $2^3$, or, in general, by $2^k$, where $k$ is the number of precrossings of the pseudoknot.

Since for pseudoknots is used the same generalized Conway notation as for virtual knots, drawings of the corresponding pseudoknots given by their Conway symbols can be produced by using WebMathematica, at the address: {\tt http://math.ict.edu.rs:8080/webMathematica/virt/virt000.jsp}

\bigskip

\noindent  \begin{tabular}{|c|c|c|} \hline
$3_1.1$ & $(i,i,i)=(i^3)$ & $\{(0_1,{6\over 2^3}),(3_1,{2\over 2^3})\}$  \\  \hline
$3_1.2$ & $(i,i,1)=(i^2,1)$ & $\{(0_1,{3\over 2^2}),(3_1,{1\over 2^2})\}$  \\  \hline
$3_1.3$ & $(i,1,1)=(i,1^2)$ & $\{(0_1,{1\over 2}),(3_1,{1\over 2})\}$  \\  \hline
\end{tabular}

\bigskip

\begin{figure}[th]
\centerline{\includegraphics[width=3.6in]{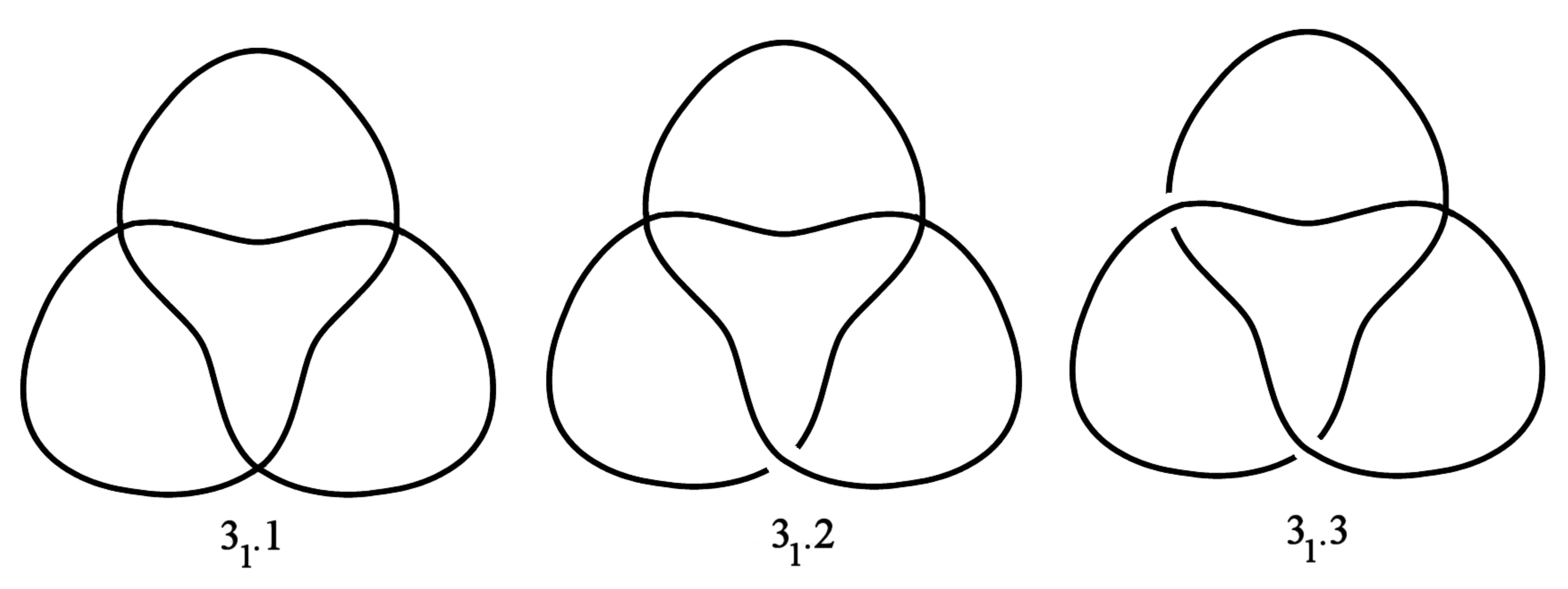}} \vspace*{8pt}
\caption{Pseudoknots $3_1.1$-$3_1.3$ derived from the trefoil knot $3_1$.
\label{f12}}
\end{figure}

\bigskip

\noindent  \begin{tabular}{|c|c|c|} \hline
$4_1.1$&$(i,i)(i,i)=(i^2)\,(i^2)$&$\{(0_1,{12\over 2^4}),(3_1,{2\over 2^4}),(4_1,{2\over 2^4})\}$\\\hline
$4_1.2$&$(i,i)(i,1)=(i^2)\,(i,1)$&$\{(0_1,{6\over 2^3}),(3_1,{1\over 2^3}),(4_1,{1\over 2^3})\}$\\\hline
$4_1.3$&$(i,i)(1,1)=(i^2)\,(1^2)$&$\{(0_1,{2\over 2^2}),(3_1,{1\over 2^2}),(4_1,{1\over 2^2})\}$\\\hline
$4_1.4$&$(i,1)(i,1)$&$\{(0_1,{3\over 2^2}),(4_1,{1\over 2^2})\}$\\\hline
$4_1.5$&$(i,1)(1,1)=(i,1)\,(1^2)$&$\{(0_1,{1\over 2}),(4_1,{1\over 2})\}$\\\hline
\end{tabular}

\bigskip

\begin{figure}[th]
\centerline{\includegraphics[width=4.8in]{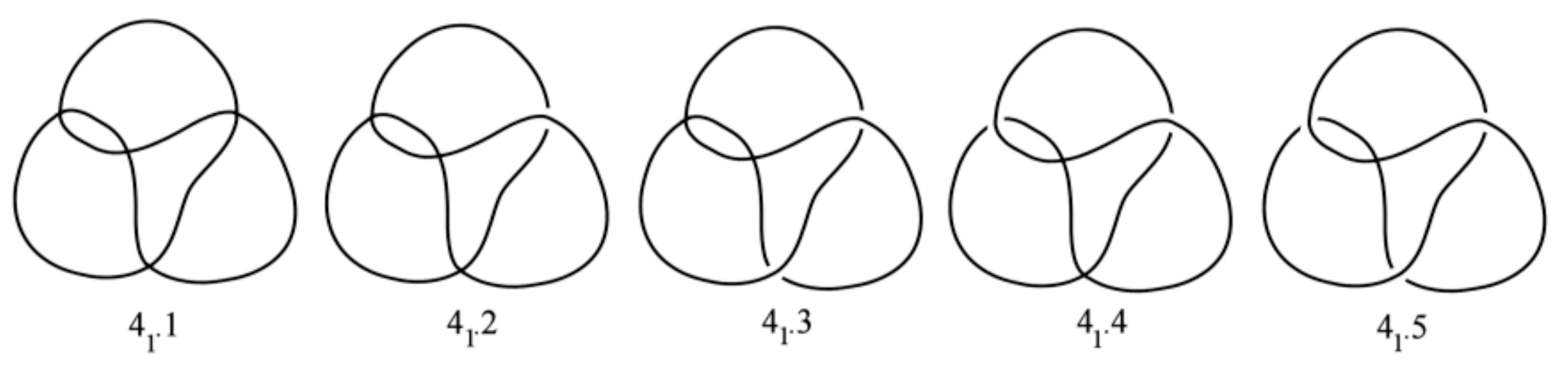}} \vspace*{8pt}
\caption{Pseudoknots $4_1.1$-$4_1.5$ derived from the figure-eight knot $4_1$.
\label{f12}}
\end{figure}

\bigskip

\noindent  \begin{tabular}{|c|c|c|} \hline
$5_1.1$&$(i,i,i,i,i)=(i^5)$&$\{(0_1,{20\over 2^5}),(3_1,{10\over 2^5}),(5_1,{2\over 2^5})\}$\\\hline
$5_1.2$&$(i,i,i,i,1)=(i^4,1)$&$\{(0_1,{10\over 2^4}),(3_1,{5\over 2^4}),(5_1,{1\over 2^4})\}$\\\hline
$5_1.3$&$(i,i,i,1,1)=(i^3,1^2)$&$\{(0_1,{4\over 2^3}),(3_1,{3\over 2^3}),(5_1,{1\over 2^3})\}$\\\hline
$5_1.4$&$(i,i,1,1,1)=(i^2,1^3)$&$\{(0_1,{1\over 2^2}),(3_1,{2\over 2^2}),(5_1,{1\over 2^2})\}$\\\hline
$5_1.5$&$(i,1,1,1,1)=(i,1^4)$&$\{(3_1,{1\over 2}),(5_1,{1\over 2})\}$\\\hline
\end{tabular}

\bigskip

\begin{figure}[th]
\centerline{\includegraphics[width=4.8in]{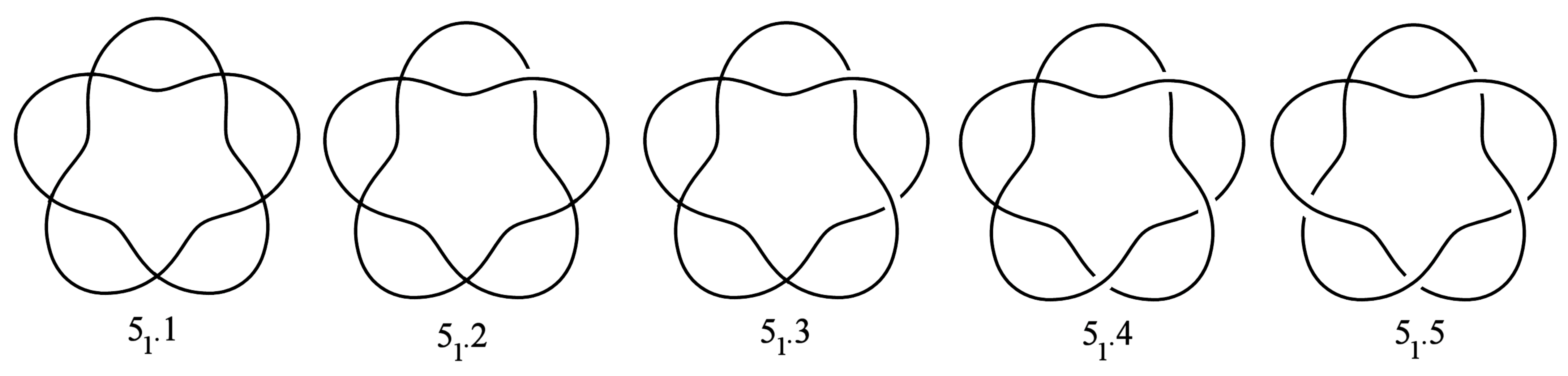}} \vspace*{8pt}
\caption{Pseudoknots $5_1.1$-$5_1.5$ derived from the knot $5_1$.
\label{f12}}
\end{figure}

\bigskip

\noindent  \begin{tabular}{|c|c|c|} \hline
$5_2.1$&$(i,i,i)(i,i)=(i^3)\,(i^2)$&$\{(0_1,{22\over 2^5}),(3_1,{6\over 2^5}),(4_1,{2\over 2^5}),(5_2,{2\over 2^5})\}$\\\hline
$5_2.2$&$(i,i,i)(i,1)=(i^3)\,(i,1)$&$\{(0_1,{11\over 2^4}),(3_1,{3\over 2^4}),(4_1,{1\over 2^4}),(5_2,{1\over 2^4})\}$\\\hline
$5_2.3$&$(i,i,i)(1,1)=(i^3)\,(1^2)$&$\{(0_1,{3\over 2^3}),(3_1,{3\over 2^3}),(4_1,{1\over 2^3}),(5_2,{1\over 2^3})\}$\\\hline
$5_2.4$&$(i,i,1)(i,1)(i^2,1)\,(i,1)$&$\{(0_1,{5\over 2^3}),(3_1,{2\over 2^3}),(5_2,{1\over 2^3})\}$\\\hline
$5_2.5$&$(i,i,1)(1,1)(i^2,1)\,(1^2)$&$\{(0_1,{1\over 2^2}),(3_1,{2\over 2^2}),(5_2,{1\over 2^2})\}$\\\hline
$5_2.6$&$(i,1,1)(i,i)=(i,1^2)\,(i^2)$&$\{(0_1,{5\over 2^3}),(3_1,{1\over 2^3}),(4_1,{1\over 2^3}),(5_2,{1\over 2^3})\}$\\\hline
$5_2.7$&$(i,1,1)(i,1)=(i,1^2)\,(i,1)$&$\{(0_1,{2\over 2^2}),(3_1,{1\over 2^2}),(5_2,{1\over 2^2})\}$\\\hline
$5_2.8$&$(i,1,1)(1,1)=(i,1^2)\,(1^2)$&$\{(3_1,{1\over 2}),(5_2,{1\over 2})\}$\\\hline
$5_2.9$&$(1,1,1)(i,i)=(1^3)\,(i^2)$&$\{(0_1,{2\over 2^2}),(4_1,{1\over 2^2}),(5_2,{1\over 2^2})\}$\\\hline
$5_2.10$&$(1,1,1)(i,1)=(1^3)\,(i,1)$&$\{(0_1,{1\over 2}),(5_2,{1\over 2})\}$\\\hline
\end{tabular}

\bigskip

\begin{figure}[th]
\centerline{\includegraphics[width=4.8in]{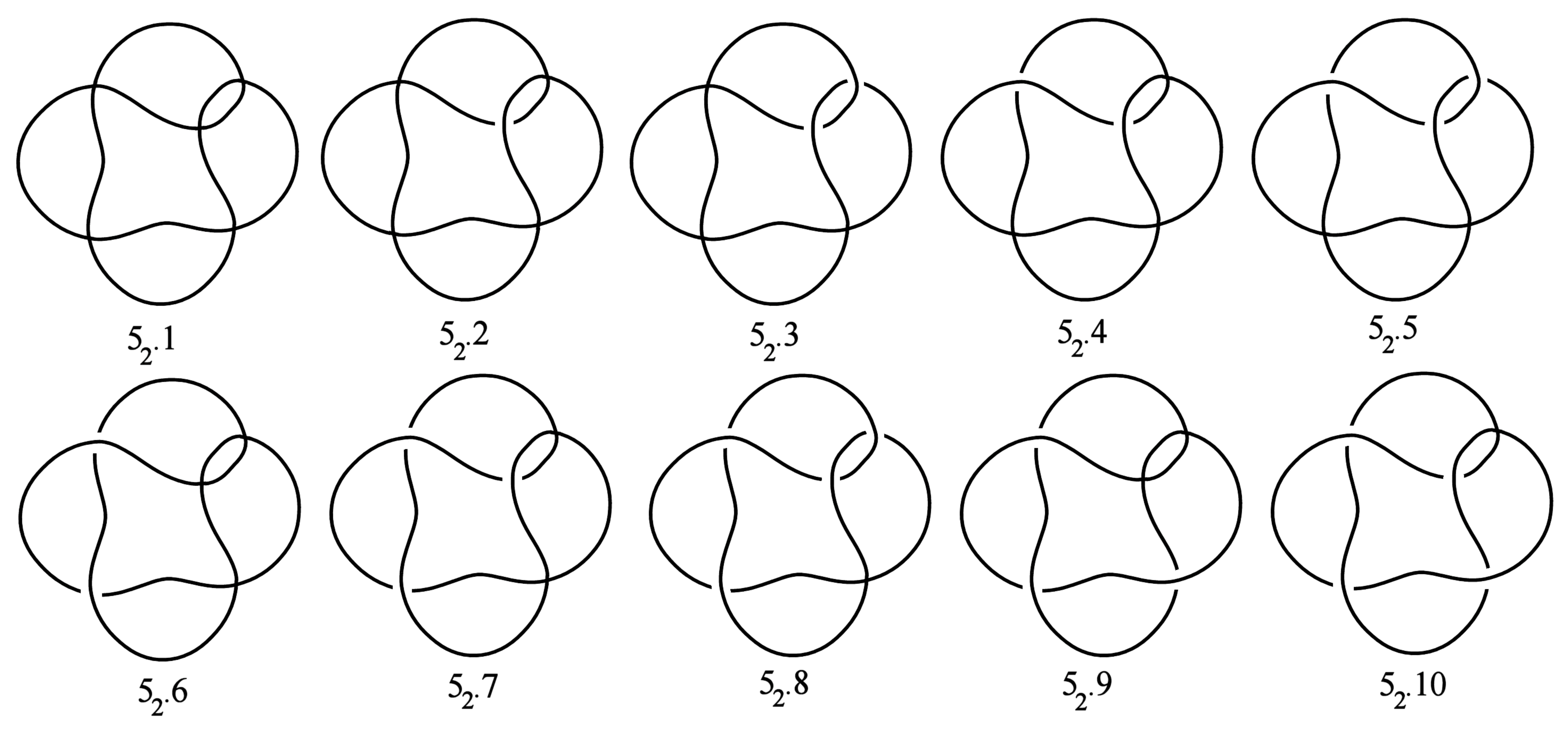}} \vspace*{8pt}
\caption{Pseudoknots $5_2.1$-$5_2.10$ derived from the knot $5_2$.
\label{f12}}
\end{figure}

\subsection*{Acknowledgments}  We are grateful to Jim Morrow and the NSF for allowing us to work on this research as a part of the University of Washington REU. We also wish to acknowledge Seattle University for summer research support. Slavik Jablan and Ljiljana Radovi\' c thank for support of the Serbian Ministry of Science (Grant No. 174012). It is also a pleasure express our gratitude to Lena Folwaczny for pointing out the history of the cosmetic crossing problem and to Roger Fenn for sparking our interest in this problem. Finally, we are indebted to Christopher Tuffley for pointing out an error in Figure 2 in the original version of this paper.

\end{document}